\documentclass[11pt]{amsart}

\usepackage[latin1]{inputenc}
\usepackage[T1]{fontenc}
\usepackage{amssymb}
\usepackage{amsmath}
\usepackage[english]{babel}
\usepackage{graphicx}
\usepackage{epsfig}
\usepackage[mathscr]{eucal}
\usepackage{color}
\usepackage{latexsym}
\usepackage{stmaryrd}
\usepackage[all]{xy}
\usepackage{url}

\theoremstyle{plain}
\newtheorem{thm}{Theorem}[section]
\newtheorem{theorem}[thm]{Theorem}
\newtheorem{prop}[thm]{Proposition}
\newtheorem{corollary}[thm]{Corollary}
\newtheorem{lem}[thm]{Lemma}
\theoremstyle{definition}
\newtheorem{definition}[thm]{Definition}
\newtheorem{question}{Question}
\theoremstyle{remark}
\newtheorem{ex}[thm]{Example}
\newtheorem{example}[thm]{Example}
\newtheorem{remark}{Remark}

\newcommand{\Pos}{\mathbf{PreOrd}} 
\newcommand{\Cat}{\mathbf{Cat}}
\newcommand{\Set}{\mathbf{Set}}
\newcommand{\llp}{\boxslash}
\AtBeginDocument{\DeclareMathSymbol{:}{\mathpunct}{operators}{"3A}}
\newcommand{\LL}{\mathcal{L}}
\newcommand{\RR}{\mathcal{R}}
\newcommand{\C}{\mathcal{C}}
\newcommand{\D}{\mathcal{D}}
\newcommand{\E}{\mathscr{E}}
\newcommand{\CC}{\mathbb{C}}
\newcommand{\initial}{\varnothing}
\newcommand{\I}{\mathscr{I}}
\renewcommand{\P}{\mathscr{P}}
\renewcommand{\S}{\mathcal{S}}
\DeclareMathOperator{\Hom}{Hom}
\DeclareMathOperator{\ob}{ob}

\DeclareMathOperator{\iso}{iso}
\DeclareMathOperator{\Ho}{Ho}
\newcommand{\defeq}{\stackrel{\mathrm{def}}{=}}
\newcommand{\rto}{\rightarrow}

\newcommand{\fib}{\ar@{->>}}
\newdir{ (}{{}*!/-5pt/\dir^{(}}
\newcommand{\cofib}{\ar@{ (->}} 

\begin{document}

\author{Jean-Marie Droz and Inna Zakharevich}



\title{Model categories with simple homotopy categories}

\begin{abstract}
  In the present article we describe constructions of model structures on
  general bicomplete categories. We are motivated by the following question:
  given a category $\C$ with a suitable subcategory $w\C$, when is there a model
  structure on $\C$ with $w\C$ as the subcategory of weak equivalences?  We
  begin exploring this question in the case where $w\C = F^{-1}(\iso \D)$ for
  some functor $F:\C\rightarrow \D$.  We also prove properness of our
  constructions under minor assumptions and examine an application to the
  category of infinite graphs.
\end{abstract}
\maketitle


\section{Introduction}

Model categories are very useful structures for analyzing the homotopy-theoretic
properties of various problems.  However, constructing these structures is
generally difficult; often, only the weak equivalences arise naturally, and much
effort must be expended to find compatible sets of cofibrations and fibrations.
(For examples of this, see \cite{hirschhorn}, \cite{hovey99}, chapter VII of
\cite{goerssjardine}, \cite{bergner07}, or the discussion of various model
structures of spectra in \cite{mmss}.)  This paper is the first in a series
which explores the general structure of such problems.  It attempts to answer
the following question:
\begin{question}
  Given a bicomplete category $\C$, together with a subcategory $w\C\subseteq
  \C$ which is closed under two-of-three and retracts, when is there a model
  structure on $\C$ such that $w\C$ is the subcategory of weak equivalences?
\end{question}
This question is very difficult, and we do not possess a complete answer to it.
However, the study of some cases has yielded many interesting families of
examples, and we present the first few here.

Often the subcategory $w\C$ is obtained through a functor $F:\C\rightarrow \D$
by defining $w\C \defeq F^{-1}(\iso\D)$.  In this paper we address the case when
$\D$ is a preorder: a category where $|\Hom(A,B)| \leq 1$ for all objects
$A,B\in \D$.  Although it turns out that we cannot answer this question in full
generality even with this simplification, we answer it in the following three
cases:
\begin{enumerate}
\item $F$ has a right adjoint which is a section.
\item $F:\C\rightarrow \E$, where $\E$ is the category with two objects and one
  noninvertible morphism between them.
\item $F=R_\C$, where $R_\C$ is the universal functor from $\C$ to a preorder.
\end{enumerate}
In fact, it turns out that the methods which allow us to answer these questions
answer more general questions than the one asked here.  For example, the
construction which gives the model structure in case (2) can also be used to
construct a model structure where the noninvertible weak equivalences are the
preimage of only one of the objects.  Whenever possible we state the results we
obtain in full generality, only applying them to the case when $\D$ is a
preorder when necessary.

We will spend the majority of our time on the third type of model structure, as
it is the one with the most interesting applications.  It generalizes a model
structure on the category of finite graphs constructed in \cite{droz12}, and in
this case gives an interesting homotopy-theoretic perspective on what the
notion of a ``core'' for an infinite graph should be.  (This will be discussed
in Section~\ref{infinite-cores}.)

The following theorem sums up the main results of the paper.
\begin{theorem}
  Let $F:\C\rightarrow \D$ be a functor as described in cases 1-3.  There exists
  a model structure on $\C$ such that the weak equivalences are $F^{-1}(\iso
  \D)$.  This model structure is left proper.
\end{theorem}

A recurring example in this paper is the category of semi-simplicial sets.  This
has as objects functors $\Delta^{op}_{inj} \rto \Set$, where $\Delta_{inj}$ is
the category of nonempty ordered sets and injections between them.  Any
simplicial set is also a semi-simplicial set, and the geometric realization of a
semi-simplicial set is homotopy equivalent to the geometric realization of the
original simplicial set.  However, this category is not a model for the homotopy
theory of topological spaces, in the sense that it does not have a model
structure Quillen equivalent to the model structure on topological spaces.  In
this paper we will show that it does have several intruiguing model structures
on it, including several where the dimension of a semi-simplicial set is a
homotopy invariant.  For more details, see Examples~\ref{semi-simpl-dim},
\ref{semi-simplicial2} and~\ref{semi-simplicial}.

As another application of this theorem we consider the model structure
constructed in \cite{droz12} for the category of finite graphs.  This model
structure is interesting in that it gives a homotopy-theoretic expression of a
combinatorial invariant: two graphs are weakly equivalent if and only if they
have the same core.  The theorem allows us to construct an analogous structure
on the category of infinite graphs, and thus gives a possible generalization of
the notion of ``core'' to the context of infinite graphs.  The notion of core
for infinite graphs is not agreeed upon, although several candidates are
defined; our new notion of core does not agree with any of the existing
candidates for the notion of the core of an infinite graph.

The organization of this paper is as follows.  Section 2 discusses model
structures and some categorical preliminaries necessary for the paper.  Sections
3-5 discuss cases 1-3 in detail.  Finally, section 6 analyzes the implications
that the model structure from section 5 has for the notion of a core for
infinite graphs.

\subsection*{Notation and terminology}
We will say that a category is {\it bicomplete} if it has all finite limits
and colimits.  We only use finite limits and colimits instead of the usual
assumption of small limits and colimits as we do not need to use the techniques
of cofibrant generation for constructing our model structures.  Thus categories
such as the category of finite graphs can be given model structures in our
examples. 

A {\it preorder} is a category where for all objects $A$ and $B$, $|\Hom(A,B)|
\leq 1$.

We write $\initial$ for the initial object in a category and $*$ for the
terminal object.

\section{Preliminaries}

\subsection{Weak factorization systems}

\begin{definition}
  In a category $\mathcal{C}$, we say that the morphism $f:A\to B$ {\it has the
    left lifting property with respect to} the morphism $g:C\to D$ if for any
  commutative diagram of solid arrows
  $$\xymatrix{
    A \ar[r]\ar[d]_{f} & C \ar[d]^{g} \\
    B \ar[r] \ar@{-->}[ur]^{h} & D }$$ there is a morphism $h$ which makes the
  complete diagram commutative.  We will write $f\llp g$ if $f$ has the left
  lifting property with respect to $g$.

  For any class of morphisms $S$, we define 
  \begin{align*}
    S^\llp &= \{g\in \C\,|\, f\llp g\hbox{ for all } f\in S\}, \\
    {}^\llp S &= \{f\in \C \,|\, f\llp g \hbox{ for all }g\in S\}.
  \end{align*}
\end{definition}

Note that for any set $S$, the sets $S^\llp$ and $^\llp S$ are closed under
retracts.

\begin{definition}
  A {\it maximal lifting system} $(\LL,\RR)$ in a category $\mathcal{C}$ is a
  pair of classes of morphisms, such that $\LL = {}^\llp\RR$ and $\RR = \LL^\llp$.
\end{definition}

The following theorem is well-known; for a proof (and a more general statement),
see \cite[14.1.8]{mayponto}.

\begin{thm}[Folklore] \label{wfs} 
  If $(\LL,\RR)$ is a maximal lifting system in a category $\C$, $\LL$ and $\RR$ contain all
  isomorphisms and are closed under composition and retraction. Moreover, $\LL$
  is closed under coproducts and pushouts along morphisms in $\C$, and $\RR$ is closed 
  under products and pullbacks along morphisms in $\C$.
\end{thm}

We recall the definition of a weak factorization system.  For more on weak
factorization systems, see for example \cite{adamek02} or \cite[Section 11]{riehl14}.

\begin{definition}
  A {\it weak factorization system} $(\LL,\RR)$ in the category $\C$ is a
  maximal lifting system such that any morphism in $\C$ can be factored as $g
  \circ f$ with $f\in \LL$ and $g\in \RR$.
\end{definition} 

From this point on, we will write WFS for ``weak factorization system.''  The
following is a well-known result for recognizing WFSs; for a proof, see
\cite[14.1.13]{mayponto}.

\begin{lem}[Folklore] \label{seeWFS}
  If $(\LL,\RR)$ is a pair of classes of morphisms in a category $\C$ such that 
  \begin{enumerate}
  \item $f\llp g$ for all $f\in \LL$ and $g\in \RR$,
  \item all morphisms $f\in \C$ can be can be factored as $f_R\circ f_L$, where
    $f_R\in \RR$ and $f_L\in \LL$, and
  \item $\LL$ and $\RR$ are closed under retracts,
  \end{enumerate}
  then $(\LL,\RR)$ is a WFS.
\end{lem}

As an example of how lifting properties can classify properties of morphisms, we
present the following characterization of retractions and sections.

\begin{definition}
  A morphism $r:A\rightarrow B$ in a category is called a {\it retraction} if it
  is possible to factorize the identity of $B$ as $1_B=r s$ for some morphism
  $s$. Dually, a morphism $s:A\rightarrow B$ is called a {\it section} if it is
  possible to factorize the identity of $A$ as $1_A=r s$ for some morphism $r$.
\end{definition}

\begin{lem}
  \label{initialTerminalLifting}
  The class of retractions is exactly $\{\initial \rightarrow A\,|\, A\in
  \C\}^\llp$. Dually, the class of sections is exactly ${}^\llp\{A\rightarrow *
  \,|\, A\in \C\}$.
\end{lem}



\subsection{Model categories}

We now recall the definition of a model structure on a category.  Instead of
using the most traditional approach \cite{hirschhorn, hovey99} we use an
equivalent axiomatization using WFSs.  For a more thorough treatment of model
categories along these lines, see for example \cite[Section 14.2]{mayponto} or
\cite[Section 11.2]{riehl14}.

\begin{definition}
  \label{shorterAxioms}
  A {\it model structure} $\CC$ on a bicomplete category $\C$ is a tuple of
  three subcategories of $\C$ called the {\it weak equivalences} ($\CC_{we}$),
  the {\it cofibrations} ($\CC_{cof}$) and the {\it fibrations}
  ($\CC_{fib}$). Those three sets should satisfy the following axioms.
  \begin{description}
  \item[WFS]  The pairs 
    \[(\CC_{cof},\CC_{fib}\cap \CC_{we}) \qquad (\CC_{cof}\cap
    \CC_{we},\CC_{fib})\] are WFSs.
  \item[2OF3] For composable morphisms $f$ and $g$, if two of the
      morphisms $f$, $g$ and $gf$ are weak equivalences, then so is the third.
  \end{description}
  We call a morphism which is both a cofibration (resp. fibration) and a weak
  equivalence an {\it acyclic cofibration} (resp. {\it acyclic fibration}).
\end{definition}

One nontrivial consequence of these axioms is that $\CC_{we}$ is closed under
retracts.  This result is due to Tierney, but we could not find it in his
writings; for a proof of this lemma, see \cite[14.2.5]{mayponto} or
\cite[11.2.3]{riehl14}.

\begin{lem}[Tierney] 
	$\CC_{we}$ is closed under retracts.
\end{lem}

The following two lemmas will be used below to construct model structures.  We
omit the proofs, as they are simple definition checks.

\begin{lem} \label{constructModel} Given a bicomplete category $\C$, together
  with subcategories $\tilde f\C \subseteq w\C \subseteq \C$, where $\tilde f\C$
  is closed under pullbacks and $w\C$ satisfies (2OF3), we define
  \begin{align*}
   \CC_{we} = w\C  \qquad
   \CC_{cof} = {}^\llp \tilde f\C \qquad 
   \CC_{fib} = (\CC_{cof}\cap \CC_{we})^\llp.
  \end{align*}
  If $(\CC_{cof},\tilde f\C)$ and $(\CC_{cof}\cap \CC_{we},\CC_{fib})$ are WFSs and $\CC_{we}\cap \CC_{fib} = \tilde f\C$, then
  $(\CC_{we}, \CC_{cof}, \CC_{fib})$ is a model structure on $\C$.
\end{lem}

\begin{lem} \label{allcofibs} Given a bicomplete category $\C$ together
  with a subcategory $w\C$ which satisfies (2OF3), we define
  \[\CC_{we} = w\C \qquad \CC_{cof} = \C \qquad \CC_{fib} = \CC_{we}^\llp.\]
  If $(\CC_{we}, \CC_{fib})$ is a WFS then $\CC$ is a
  model structure on $\C$.
\end{lem}

We conclude the discussion of model categories by recalling the definition of a
proper model category.

\begin{definition}
  A model structure is {\it left (resp. right) proper} if the pushout
  (resp. pullback) of a weak equivalence along a cofibration (resp. fibration)
  is always a weak equivalence.
\end{definition}

\subsection{Splitting and disjoint coproducts}

Our last topic in this section is splitting and disjoint coproducts.  We
present several examples, as these notions interact in nontrivial ways.  We
write $A \sqcup B$ for the coproduct of $A$ and $B$.

\begin{definition}
  A category is said to have {\it splitting coproducts} if for any morphism
  $f:X\rightarrow A\sqcup B$ there exist objects $X_L$ and $X_R$ such that $X
  \cong X_L \sqcup X_R$, and morphisms $f_L: X_L \rto A$ and $f_R: X_R \rto B$
  such that $f \cong f_L \sqcup f_R$.  (Although $X_L$ and $X_R$ depend on $f$,
  we do omit it from the notation.)
\end{definition}


\begin{definition}
  A category is said to have {\it disjoint coproducts} if, for any coproduct
  $A\sqcup B$, the natural injections $i_1:A \rto A\sqcup B$ and $i_2:B \rto A
  \sqcup B$ are monic, and the following three squares are pullback squares:
  \[\xymatrix{A \ar[r]^{1_A} \ar[d]_{1_A} & A \ar[d]^{i_1} \\ A \ar[r]^-{i_1} &
    A\sqcup B} \qquad
  \xymatrix{\initial \ar[r] \ar[d] & B \ar[d]^{i_2} \\ A \ar[r]^-{i_1} &
    A \sqcup B} \qquad
  \xymatrix{B \ar[r]^{1_B} \ar[d]_{1_B} & B \ar[d]^{i_2} \\ B \ar[r]^-{i_2} &
    A \sqcup B} 
  \]
\end{definition}

\begin{ex} We present examples of how these definitions interact.
\begin{enumerate}
\item The categories of sets and graphs both have splitting coproducts and disjoint coproducts.
\item The category of vector spaces and linear maps over $\mathbb{R}$ has
  disjoint coproducts but not splitting coproducts. 
\item The category of pointed finite sets has both disjoint and
  splitting coproducts.  
\item The lattice
  \[\xymatrix@R=1.5em{& & B \ar[rd] \\ \initial \ar[r] & X \ar[ru] \ar[rd] & & {*} \\ &
    & A \ar[ru]}\]
  has splitting but not disjoint coproducts, as $* = A \sqcup B$ but the
  pullback of the two inclusions is $X$.
\item The lattice
  \[\xymatrix{& & B \ar[rd] \\ \initial \ar[r] & X \ar[ru] \ar[r] \ar[rd] & C
    \ar[r] & {*} \\
    & & A \ar[ru] 
  }\]
  has neither disjoint nor splitting coproducts.  Just as in (4) it does not
  have disjoint coproducts, and it does not have splitting coproducts because $C
  \rto * = A \sqcup B$ cannot be written as $(Y \rto A) \sqcup (Z \rto B)$ for
  any $Y$ or $Z$ in the lattice.
\end{enumerate}
\end{ex}

\begin{lem} \label{lem:no-trans} Suppose that $\C$ has disjoint coproducts, and
  suppose that there exists a morphism $f:A \rightarrow B$ such that the diagram
  \[\xymatrix{& B\ar[d]^{i_2} \\ A \ar[ru]^f \ar[r]^{i_1} & A\sqcup B}\]
  commutes.  Then $A \cong \initial$.
\end{lem}

\begin{proof}
  We have the following diagram, where both squares are pullback squares:
  \[\xymatrix{A \ar[r]^{1_A} \ar[d] & A \ar[d]^f \ar@/^{5ex}/[dd]^{i_1}\\
    \initial \ar[r] \ar[d] & B \ar[d]^{i_2} \\ A \ar[r]^-{i_1} & A\sqcup B}
  \]
  If we consider the composite pullback, we need to take the pullback of $i_1$
  along $i_1$, which (as $\C$ has disjoint coproducts) is $A$ with the identity
  morphism.  Thus the composite down the left must be the identity on $A$, and
  $A\cong \initial$, as desired.
\end{proof}

\section{The case of a functor with an adjoint section}

In many cases where a model category is required, the subcategory of weak
equivalences is given as the preimage of the isomorphisms under a functor.  In
this section we explore the question of how much extra structure on the functor
is required to show that the model structure exists directly from the existence
of the functor.

Let $\C$ be a bicomplete category and suppose that $F:\C\rightarrow \D$ is a
functor with a right adjoint $G:\D\rightarrow \C$ such that the counit
$\epsilon:FG \rto 1_\D$ is a natural isomorphism; in this case, we say that $G$
is a {\it section} of $F$.  We would like to define a model structure on $\C$
such that $\D = \mathrm{Ho}\,\C$ and $F$ is the localization functor.

\begin{prop} \label{prop:pullup} 
  Suppose that $F: \C \rto \D$ is a functor with a section $G: \D \rto \C$.
  We define three subcategories of $\C$ by 
  \[\CC^{adj}_{we} = F^{-1}(\iso \D) \qquad \CC^{adj}_{cof} = \C \qquad \CC^{adj}_{fib} = (\CC^{adj}_{we})^\llp.\]
  Suppose that either
  \begin{itemize}
  \item[(a)] $\CC^{adj}_{we}$ is closed under pullbacks along $G(\D)$, or
  \item[(b)] for any $f:A\rightarrow B$, $A \rightarrow B\times_{GF(B)} GF(A)$
    is in $\CC^{adj}_{we}$.
  \end{itemize}
  Then these three subcategories form a left proper model structure.  If $\D$ is
  bicomplete we can consider $\D$ to be a model category where the weak
  equivalences are the isomorphisms and all morphisms are both cofibrations and
  fibrations.  In this case, $(F,G)$ is a Quillen equivalence.
\end{prop}

\begin{proof}
  First, notice that the image of $G$ is inside $\CC^{adj}_{fib}$.  A
  commutative square
  \[\xymatrix{A \ar[r] \ar[d]_f^\sim & G(X) \ar[d]^{G(p)} \\ B \ar[r] & G(Y) }\]
  has a lift because $F(f)$ is an isomorphism in the adjoint square.  Thus a lift
  exists in the original square and $G(p) \in (\CC^{adj}_{we})^\llp = \CC^{adj}_{fib}$.

  Now we check that we have a model structure on $\C$.  As $\CC^{adj}_{we}$
  clearly satisfies (2OF3), by Lemma~\ref{allcofibs} we only need to check that
  $(\CC^{adj}_{we}, \CC^{adj}_{fib})$ is a WFS.  We use
  Lemma~\ref{seeWFS}.  From the definitions we know that $\CC^{adj}_{we}$ and
  $\CC^{adj}_{fib}$ are closed under retracts, so conditition (3) holds.  As
  $\CC^{adj}_{fib} = (\CC^{adj}_{we})^\llp$, condition (1) holds as well.

  It remains to check condition (2): any morphism can be factored as a weak
  equivalence followed by a fibration.  Let $f:A\rightarrow B$ be any morphism,
  and consider the diagram
  \[\xymatrix{A \ar@/^3ex/[rrd]^{\eta_A} \ar@/_3ex/[rdd]_f \ar@{-->}[rd]^i\\
    & B\times_{GF(B)} GF(A) \ar[r] \fib[d] & GF(A) \fib[d]^{GF(f)}\\
    & B \ar[r]^{\eta_B} & GF(B)}\]
  where $\eta: 1_\C \rto GF$ is the unit of the adjunction.  We need to show
  that $i$ is a weak equivalence.  If (b) holds then this is true by assumption.
  On the other hand, since $\epsilon$ is a natural isomorphism we know that as
  $F(\eta_X)$ is an isomorphism for all $X$, $\eta_A$ and $\eta_B$ are weak
  equivalences.  If (a) holds then $B\times_{GF(B)}GF(A)\rightarrow GF(A)$ is
  also a weak equivalence, and by (2OF3) $i$ is, as well.  In either case we
  have a factorization, as desired.  Thus $(\CC^{adj}_{we}, \CC^{adj}_{fib})$ is
  a WFS, as desired.

  It remains to check that $\CC^{adj}$ is left proper.  Consider any diagram
  \[\xymatrix{C & A \ar[l]_g \cofib[r]^i & B}\]
  where $g$ is a weak equivalence.
  Since $F$ is a left adjoint, $F(B\rightarrow B\cup_A C) = F(B) \rightarrow
  F(B)\cup_{F(A)} F(C)$, which is an isomorphism because $F(g)$ is an
  isomorphism.  Thus the pushout of $g$ along $i$ is a weak equivalence, as
  desired.

  We need to check that if $\D$ is bicomplete then $(F,G)$ is, indeed, a Quillen
  equivalence.  Clearly $F$ preserves cofibrations and acyclic cofibrations, so
  we have a Quillen pair.  It remains to show that $FA \rightarrow X$ is an
  isomorphism if and only if $A \rightarrow GX$ is a weak equivalence.  But $FA
  \rightarrow X$ is an isomorphism if and only if $GFA \rightarrow GX$ is an
  isomorphism (as $FG \simeq 1_\D$), which is an isomorphism if and only if $A
  \stackrel{\sim}{\rightarrow} GFA \rightarrow GX$ is a weak equivalence, as
  desired.
\end{proof}




\begin{remark}
  One may ask whether the model structure constructed in
  Proposition~\ref{prop:pullup} is right proper.  Unfortunately, we could not
  resolve that question.  The model structures constructed in
  Examples~\ref{top-conn-comp} and \ref{semi-simpl-dim} below are right proper, but we
  could not find a proof that this is generally the case.
\end{remark}

Conditions (a) and (b) are a bit annoying, as we do not have a conceptual
explanation of why they are necessary; they are assumed simply because they are
needed in the proof.  Morally speaking, they should correspond to the fact that
$\D$ is a much simpler category than $\C$, and thus that we don't need any more
information about the problem than just the structure of $\D$.  For very simple
$\D$ this is the case:
\begin{corollary} \label{D-preorder-adj}
  If $\D$ is a preorder then condition (b) always holds.  Thus for any functor
  $F:\C \rto \D$ with a section we have a model structure $\CC^{adj}$ with
  $\CC^{adj}_{we} = F^{-1}(\iso \D)$.
\end{corollary}

\begin{proof}
  We need to show that $F(i:A \rto B\times_{GF(B)} GF(A)) \in \iso \D$ if $\D$ is
  a preorder.  Let $\pi_2: B \times_{GF(B)}GF(A) \rto GF(A)$ be the projection morphism;
  then we know that $\pi_2 i = \eta_A$.  Since $FG = 1_\D$, in particular we
  know that $F(\eta_A) = 1_{F(A)}$; thus $F(\pi_2 i) = F(\pi_2) F(i)=
  1_{F(A)}$.  Since $\D$ is a preorder, this means that $F(i)$ is an
  isomorphism, as desired.
\end{proof}

We present a couple of examples of model structures constructed using this
theorem. 

\begin{example} \label{top-conn-comp} Let $\pi_0:\mathbf{Top} \rto \Set$ be the
  functor which takes a topological space to the set of its connected
  components.  This functor has a right adjoint $-^\delta$ which endows a set
  with the discrete topology.  To check that a model structure exists with weak
  equivalences equal to $\pi_0^{-1}(\iso \Set)$ (the morphisms which induce
  bijections between connected components) we will show that condition (b)
  holds.  Let $f:A \rto B$ be a continuous map of spaces.  Write $A =
  \coprod_{i\in I} A_i$ and $B = \coprod_{j\in J} B_j$, with $A_i$ and $B_j$
  connected; by an abuse of notation, write $f:I \rto J$ for the induced map on
  connected components.  Thus $\pi_0(A)^\delta = I$ and $\pi_0(B)^\delta = J$, and
  \[B\times_{\pi_0(B)^\delta} \pi_0(A)^\delta \cong \coprod_{i\in I} B_{f(i)}\]
  with the map $A \rto B\times_{GF(B)} GF(A)$ sending $A_i$ to $B_{f(i)}$ by
  $f$.  This induces a bijection on connected components, so it is a weak
  equivalence, as desired.  
\end{example}

\begin{example} \label{semi-simpl-dim} Let $s_{inj}\Set$ be the category of {\it
    semi-simplicial sets} as defined in the introduction.  For $X\in
  s_{inj}\Set$, let $\dim X$ be the smallest integer $n$ such that
  $X(0<\cdots<k)$ is empty for all $k>n$; if such an integer does not exist then
  we write $\dim X = \infty$.  Let $\mathbf{Z}^+_{\geq 0}$ be the category with
  objects nonnegative integers and $\infty$ and morphisms $n\rto m$ if $n \leq
  m$.  We have a functor $F:s_{inj} \Set \rto \mathbf{Z}^+_{\geq 0}$ given by
  taking $X$ to $\dim X$.  This functor has right adjoint $G$ which takes $n$ to
  $D_n: \Delta^{op}_{inj} \rto \Set$ defined by
  \[D_n(m) = \begin{cases} * & \hbox{if } m\leq n \\ \emptyset &
    \hbox{otherwise}.\end{cases}\]
  Note that $FG = 1$, and thus by Corollary~\ref{D-preorder-adj} we have a model
  structure on $s_{inj}\Set$ where $X$ and $Y$ are weakly equivalent exactly
  when they have the same dimension.
\end{example}

\section{Model structures from simple preorders}

In the previous section, we showed that if $\D$ is a preorder then for any
functor $F:\C\rto \D$ with a section there exists a model structure $\CC$ on
$\C$ such that $\CC_{we} = F^{-1}(\iso\D)$.  This result is not completely
satisfying, however, as the condition that $F$ has a right adjoint section is
much too strong to hold in general.  Thus in this section we will try to analyze
this problem with the (equally strong but) different assumption that the
structure of $\Ho\CC$ is very simple.  In the course of this exploration we
actually construct several model structures whose homotopy categories are {\it
  not} preorders; we include them in the discussion as well, since their proofs
are identical, and they give an interesting family of model structures.

The simplest that $\Ho\CC$ could be, of course, is if it is equivalent to the
trivial category.  Such a model structure always exists, by setting the weak
equivalences and the cofibrations to be all morphisms, and the fibrations to be
the isomorphisms.  This resolved, we consider the second-simplest case, when
$\Ho\CC$ has two objects and one morphism between them.  Let $\E$ be the
category with two objects, $\initial$ and $*$, and one non-identity morphism
$\initial \rightarrow *$.  In this case the model structure $\CC$ divides
objects of $\C$ into ``big'' objects and ``small'' objects, but does not
distinguish between different ``big'' or different ``small'' objects.

\begin{definition}
  A {\it cut} of a category $\C$ is a functor $F:\C\rightarrow \E$; such a cut
  is called {\it trivial} if $\C = F^{-1}(\initial)$ or $\C= F^{-1}(*)$.  Given
  any cut $F$, we define $\I_F = F^{-1}(\initial)$, $\P_F = F^{-1}(*)$.  When
  $F$ is clear from context we omit the subscript from the notation.
\end{definition}



We start with a more general construction which will give us three model
structures associated to any cut.  These three model structures will either (a)
classify objects into ``big'' and ``small'', (b) distinguish between all
``small'' objects but have all ``big'' objects be equivalent, or (c) distinguish
between all ``big'' objects but have all ``small'' objects be equivalent.

\begin{prop} \label{double-cut}
  Let $\E'$ be the total order with three objects, $\initial \rightarrow E
  \rightarrow *$.  Suppose that $\C$ has a ``double cut'', a functor
  $F:\C\rightarrow \E'$.  Then we have a model structure $\CC^F$ on $\C$ given by 
  \[\CC^F_{cof} = \{A\rightarrow B
  \,|\, F(B) \neq \initial\}\cup \iso \C \qquad \CC^F_{fib} = \{X \rightarrow Y \,|\, F(X)
  \neq *\}\cup \iso \C,\]
  and
  \[\CC^F_{we} = F^{-1}\{1_\initial,1_*\}\cup\iso \C.\] 
\end{prop}

\begin{proof}
  We need to check the axioms of a model structure.  $\CC^F_{we}$ clearly
  satisfies (2OF3), so we focus on (WFS).  We will prove that $(\CC^F_{cof},
  \CC^F_{we}\cap\CC^F_{fib})$ is a WFS; the other one will
  follow by duality.  We use Lemma~\ref{seeWFS}.  As all three of the above
  classes are closed under retracts, condition (3) is satisfied.  Note that
  \[\CC^F_{we}\cap \CC^F_{fib} = F^{-1}(1_\initial)\cup \iso \C.\] 
  Thus we can say that a noninvertible morphism $f:X\rightarrow Y$ is an acyclic
  fibration when $F(Y) = \initial$, and we see that any morphism is either a
  cofibration or an acyclic fibration.  In such situations factorizations
  trivially exist, and condition (2) is satisfied.  Thus it remains to check
  condition (1).

  Let $f:A \rto B\in \CC^F_{cof}$ and $g:X\rightarrow Y\in \CC^F_{we}\cap
  \CC^F_{fib}$, and suppose that we have a commutative square
  \[\xymatrix{A \ar[r] \ar[d]_f & X \ar[d]^g \\ B \ar[r] & Y}\]
  If either $f$ or $g$ is an isomorphism then this square clearly has a lift, so
  we assume that neither is an isomorphism.  Then $F(g) = 1_\initial$, and in
  particular $F(Y) = \initial$.  But $F(B)\neq \initial$, and thus we cannot
  have a morphism $B\rto Y$.  Contradiction.  Thus in any such square either $f$
  or $g$ must be an isomorphism and $f\llp g$.  So condition (1) holds, and
  $(\CC^F_{cof}, \CC^F_{we}\cap \CC^F_{fib})$ is a WFS, as
  desired.
%

\end{proof}

We now use this proposition to construct the model structures associated to a
cut.

\begin{corollary} \label{balanced-model} Given any cut $F$ of a bicomplete $\C$
  we define the model structure $\CC^{bF}$ (the {\it balanced} model structure
  associated to $F$) on $\C$ by
  \begin{eqnarray*}
    &\CC^{bF}_{we} = \{A \rightarrow B\,|\, B\in \I \hbox{ or } A\in \P\} \cup \iso \C =
    F^{-1}(\iso \E)\\ 
    &\CC^{bF}_{cof} = \{A\rightarrow B\,|\, B\in \P\} \cup \iso \C \qquad \CC^{bF}_{fib} = \{A
    \rightarrow B\,|\, A\in \I\} \cup \iso\C.
  \end{eqnarray*}
  This model structure is both left and right proper.
\end{corollary}

$\CC^{bF}$ is the model structure that can distinguish between ``big'' and
``small'' objects, but does not detect any other differences.

\begin{proof}
  Let $F'$ be the double cut defined by composing $F$ with the functor
  $\E\rightarrow \E'$ taking $\initial$ to $\initial$ and $*$ to $*$.  Applying
  Proposition~\ref{double-cut} to $F'$ we get the desired structure.


  As the model structure is self-dual, it suffices to show that it is left
  proper.  Suppose that we have a diagram of noninvertible morphisms
  \[\xymatrix{C & A \ar[l] \cofib[r]^i & B.}\]
  As $i$ is a cofibration, $B\in \P_F$, and thus the pushout $B \rightarrow
  B\cup_A C\in \P_F$.  But then $B\rightarrow B\cup_A C$ is a weak equivalence,
  as desired.
\end{proof}

\begin{remark}
  Note that we didn't use the fact that $A \rto C$ is a weak equivalence, so in
  fact the pushout of any morphism along a noninvertible cofibration must be a
  weak equivalence.  This reflects that the balanced model structure is not very
  discriminating.
\end{remark}

Thus for any cut $F:\C \rto \E$ we can construct a model structure with homotopy
category equivalent to $\E$.  Note, however, that $F$ need not be a Quillen
equivalence, as it does not necessarily have an adjoint.  For example, consider
the category
\[\xymatrix{& A \ar[r] & B \ar[rd] \\ \initial \ar[ru] \ar[rd] & & & {*} \\
  & C \ar[r] & D \ar[ru]}\]
and define $F$ to map $\initial$, $A$ and $B$ to $\initial$ and $C$, $D$ and $*$
to $*$.  Then $F$ does not preserve either pullbacks or pushouts, so it is not a
right or a left adjoint.


\begin{corollary} \label{right-model} Given any cut $F$ of a bicomplete $\C$ we
  have a model structure $\CC^{rF}$ on $\C$ given by
  \begin{eqnarray*}
    &\CC^{rF}_{we} = \P\cup\iso\C \qquad
    \CC^{rF}_{cof} = \C \qquad \CC^{rF}_{fib} = (\CC^{rF}_{we})^\llp.
  \end{eqnarray*}
  This model structure is left proper.  As the definition of a cut is self-dual,
  we also have a dual model structure $\CC^{\ell F}$ where all morphisms are
  fibrations and the noninvertible weak equivalences are morphisms in $\I$; this
  model structure is right proper.
\end{corollary}

$\CC^{rF}$ can distinguish between all objects in $\I$ (the ``small'' objects),
but collapses all objects in $\P$ to a single one.

\begin{proof}
  We construct $\CC^{rF}$ by composing the cut with the functor $\E\rto \E'$
  which takes $\initial$ to $E$ and $*$ to $*$ and taking the model structure
  constructed in Proposition~\ref{double-cut}.  As
  $(\CC^{rF}_{we},\CC^{rF}_{fib})$ is a WFS, we know that
  $\CC^{rF}_{we}$ is closed under pushouts.  Thus $\CC^{rF}$ is left proper.

  However, $\CC^{rF}$ does not have to be right proper.  Let $\C$ be the category
  \[\xymatrix{\initial \ar[r] \ar[d] & A \ar[d] \\ B \ar[r] & {*}}\]
  and let $F$ be the cut that takes $\initial$ and $A$ to $\initial$ and $B$ and
  $*$ to $*$.  Then the only nontrivial weak equivalence is $B\rightarrow *$,
  and $A$ is a fibrant object.  If $\CC^{rF}$ were proper we would have to have
  $\initial \rightarrow A$ be a weak equivalence, but in this model structure it
  is not.  Thus in this case $\CC^{rF}$ is not right proper, as claimed.

  The second part of the corollary follows by duality.
\end{proof}

In particular, the two examples constructed in this proof also prove the
following:
\begin{corollary}
  The model structure constructed in Proposition~\ref{double-cut} is not
  necessarily left or right proper.
\end{corollary}


Thus any cut in a category $\C$ gives at least three different (but possibly
equivalent) model structures on $\C$.  This means that any category with
uncountably many cuts has uncountably many model structures, and more generally
that any category with $\kappa$ cuts has at least $\kappa$ model structures.

\begin{example}
  Any cut of a pointed category must be trivial, which means that in a model
  structure associated to a cut, the weak equivalence are either all morphisms
  or just the isomorphisms.
\end{example}

\begin{example}
  The category $\Set$ has a single non-trivial cut, which takes the empty set to
  $\initial$ and all other sets to $*$.  All model structures on $\Set$ where
  not all morphisms are weak equivalences are Quillen equivalent to either
  $\CC^{bF}$ or $\CC^{\ell F}$. (For an enumeration of the model structures on
  $\Set$, see \cite{camarenaweb}.)

  More generally, many $\Set$-based categories (topological spaces, simlicial
  sets, etc.) have a single non-trivial cut, which gives rise to a similar
  family of model structures.  However, these do not generally cover all
  possible model structures.
\end{example}

\begin{example} \label{semi-simplicial2} The category $s_{inj}\Set$ (defined in
  Example~\ref{semi-simpl-dim}) has many different cuts; for example, for any $n$
  we have a cut $F_n$ defined by $F_n(X) = \initial$ if $\dim X \leq n$ and
  $F_n(X) = *$ otherwise.  Corollaries~\ref{balanced-model} and
  \ref{right-model} give model structures which distinguish between
  semi-simplicial sets based on their dimensions: $\CC^{bF_n}$ has $X$ and $Y$
  equivalent if $\dim X, \dim Y \leq n$ or $\dim X, \dim Y > n$, $\CC^{rF_n}$
  has $X$ and $Y$ equivalent if $\dim X, \dim Y > n$ and $\CC^{\ell F_n}$ has
  $X$ and $Y$ equivalent if $\dim X, \dim Y \leq n$.
\end{example}

It is possible for different cuts to yield equivalent model structures.  For
example, consider the category $\C$ with objects $\mathbb{R} \cup
\{\pm\infty\}$, and with a morphism $a\rightarrow b$ if $a<b$.  Let
\[F_a(b) = \begin{cases} \initial & \hbox{if } b<a \\ * &
  \hbox{otherwise}.\end{cases}\] Then $F_a$ is a cut for any finite value of
$a$; let $\CC_a$ be the model structure constructed by
Corollary~\ref{right-model} for $F_a$.  If we choose $a<a'$ then the functor
$G:\CC_a \rightarrow \CC_{a'}$ given by $G(b) = b-a'+a$ preserves both
cofibrations and weak equivalences and is clearly an equivalence of categories,
and thus gives a Quillen equivalence between $\CC_a$ and $\CC_{a'}$.

However, in many cases we can show that different cuts will yield inequivalent model structures.  
\begin{corollary} \label{diff-cuts} If a category $\C$ has a family of cuts
  $\{F_\alpha:\C\rightarrow \E\}_{\alpha\in A}$ such that if $\alpha\neq\alpha'$
  then $\I_\alpha$ and $\I_{\alpha'}$ are not equivalent categories, then $\C$
  has at least $|A|$ nonequivalent model structures.
  
  Dually, if such a family of cuts exists with $\P_\alpha\not\simeq
  \P_{\alpha'}$ for all distinct $\alpha,\alpha'\in A$ then $\C$ has at least
  $|A|$ nonequivalent model structures.
\end{corollary}

\begin{proof}
  Let $\alpha\neq \alpha'\in A$, and let $(\I_\alpha,\P_\alpha)$ and
  $(\I_{\alpha'}, \P_{\alpha'})$ be obtained from $F_\alpha$ and $F_{\alpha'}$,
  respectively.  Let $\CC_\alpha$ and $\CC_{\alpha'}$ be the model structures
  constructed by the first part of Corollary~\ref{right-model}.  A zigzag of
  Quillen equivalences between $\CC_\alpha$ and $\CC_{\alpha'}$ would give an
  equivalence of homotopy categories.  However, the homotopy category of
  $\CC_\alpha$ is $(\I_\alpha)_+$, the category $\I_\alpha$ with a new terminal
  object added.  As an equivalence must take terminal objects to terminal
  objects, an equivalence of $(\I_\alpha)_+$ with $(\I_{\alpha'})_+$ must give
  an equivalence of $\I_\alpha$ with $\I_{\alpha'}$; as these are inequivalent,
  we know that $\CC_\alpha$ and $\CC_{\alpha'}$ must be inequivalent, as desired.
  
  The dual version follows from the dual version of Corollary~\ref{right-model}.
\end{proof}

\begin{example}
  The model structures $\CC^{rF_n}$ from Example~\ref{semi-simplicial2} are all
  non-equivalent.  Let $\S_n = F_n^{-1}(\initial)$; by Corollary~\ref{diff-cuts}
  it suffices to check that these are nonequivalent.
  
  We define the {\it monic length} of a category $\C$ with a terminal object to
  be the maximum length of a chain
  \[A_0 \rto A_1 \rto \cdots \rto A_k = * \in \C\] such that each morphism is a
  noninvertible monomorphism and $A_k$ is the terminal object of $\C$; this is
  an equivalence invariant.  In $\S_n$ the terminal object is $D_n$, defined by
  \[D_n(k) = \begin{cases} * & \hbox{if } k \leq n \\ \emptyset &
    \hbox{otherwise}.\end{cases}\] All monomorphisms in $\S_n$ are levelwise
  injections, so the monic length of $\S_n$ is $n+1$, given by
  \[\emptyset \rto D_0 \rto D_1 \rto \cdots \rto D_n.\]
  As if $m\neq n$ then $m+1\neq n+1$, we see that $\S_m$ and $\S_n$ are not
  equivalent, as claimed.
\end{example}

\section{The generalized core model structure} \label{sec:gencore}

There are two motivations for the construction of the generalized core model
category.  The first is a continuation of the type of analysis given in the
previous section; however, in this case instead of taking $\D$ to be the
simplest possible preorder, we take it to be the most complicated. More
formally, we have the following definition:

\begin{definition}
  Let $\C$ be a category. We define the preorder $P(\C)$ with $\ob P(\C) = \ob
  \C$, and $\Hom_{P(\C)}(X,Y)$ equaling the one-point set if there exists a
  morphism $X\rightarrow Y\in \C$, and the empty set otherwise.  We will write
  $X\sim Y$ if $X$ is isomorphic to $Y$ in $P(\C)$.
\end{definition}

There is a canonical functor $R_\C: \C \rightarrow P(\C)$, such that any
functor $F:\C \rto \D$, where $\D$ is a preorder, factors through $R_\C$.  In
this section, we construct a model structure on $\C$ such that the weak
equivalences are $R_\C^{-1}(\iso P(\C))$.  Note that $P$ is a functor $\Cat \rto
\Pos$, which is left adjoint to the forgetful functor $U: \Pos \rto
\Cat$.\footnote{Technically, $P$ and $U$ are only functors if we restrict our
  attention to small categories; otherwise, we need to worry about the
  $2$-category structure of $\Cat$ and $\Pos$ and check that it is a
  $2$-adjunction.  However, as in the rest of this paper we are only concerned
  with the functor $R_\C$, which exists in any case, we blithely sweep these
  problems under the rug.}

The second motivation for constructing the generalized core model structure is
to generalize the construction of the core model category structure in
\cite{droz12}.  The {\it core} of a graph is the smallest retract of the graph,
and two graphs $G$ and $G'$ have isomorphic cores if and only if there exist
morphisms $f:G \rto G'$ and $g:G' \rto G$ in the category of graphs.  (For more
on cores, see \cite[Chapter 6]{godsilroyle}.)  In \cite{droz12}, Droz
constructed a model structure on the category of finite graphs where the weak
equivalences are exactly the morphisms between graphs with isomorphic cores.  It
turns out that a similar construction will work in any category, and in
particular on the category of infinite graphs.  This gives rise to an
application to infinite graph theory: an alternate definition of the core of an
infinite graph.  There is very little known about cores of infinite graphs, and
it turns out that the homotopy-theoretic perspective gives an entirely new
possible definition of a core.  For more on this, see
Section~\ref{infinite-cores}.

The main result of this section is the following:

\begin{thm}
  \label{generalizedCoreModel}
  There is a model structure $\CC^{core}$ with homotopy category $P(\C)$ on any
  bicomplete category $\C$. A morphism $f:A\rightarrow B$ is a weak equivalence
  iff $A\sim B$. The acyclic fibrations are exactly the retractions in $\C$.

  If in addition $\C$ has splitting and disjoint coproducts then this structure
  is both left and right proper.
\end{thm}

We call $\CC^{core}$ the {\it generalized core model structure} on $\C$.  Before
we begin the proof, we present a couple of examples of such model structures.

\begin{example} \label{empty-nonempty-model} Let $\Set$ be the category of sets.
  $P(\Set)$ is the category with two objects and one noninvertible morphism
  between them.  The core model structure can distinguish between empty and
  nonempty sets, but cannot distinguish between nonempty sets.  The fibrations
  are the surjective morphisms and the cofibrations are the injective morphisms.

  More generally, for many set-based categories
  (such as topological spaces, simplicial sets, etc.) the core model structure
  has as the weak equivalences all morphisms between ``nonempty'' objects.
\end{example}


\begin{example} \label{semi-simplicial} The category $s_{inj}\Set$ (defined in
  Example~\ref{semi-simpl-dim}) has a core which is more complicated than the
  core of simplicial sets.  For example, if $\dim X > \dim Y$ then there are no
  morphisms $X \rto Y \in s_{inj}\Set$, and in fact dimension is a homotopy
  invariant in the generalized core model structure, since if there is a
  morphism $X \rto Y$ and a morphism $Y \rto X$ then $\dim X = \dim Y$.
  However, unlike in Example~\ref{semi-simpl-dim}, it is not the only invariant,
  as there exist $X$ and $Y$ with $\dim X = \dim Y$ but with $X$ and $Y$ not
  isomorphic in $P(s_{inj}\Set)$.
\end{example}

\begin{proof}[of Theorem~\ref{generalizedCoreModel}.]
  We define $w\C$ to be the preimage under $R_\C$ of $\iso P(\C)$, and $\tilde
  f\C$ to be the subcategory of retractions in $\C$.  These satisfy the
  conditions of Lemma~\ref{constructModel}.  Let $\CC^{core}$ be the candidate
  constructed as in Lemma~\ref{constructModel}; we will show that it satisfies
  the necessary conditions to be a model structure.  Since $\CC^{core}_{we}$
  satisfies (2OF3) by definition, we focus on the other three conditions.
  
  First, an observation: suppose that $f:A \rightarrow B$ is any morphism in
  $\C$.  Then in $\CC^{core}$, the canonical projection $p_1:A\times B \rightarrow A$
  is an acyclic fibration, and the canonical inclusion $i_1:B \rightarrow B\sqcup
  A$ is a cofibration and a weak equivalence.  The first follows trivially from
  the definition of acyclic fibration, since $f$ and $1_A$ give a morphism $A
  \rightarrow A\times B$ which is a section of $p_1$.  For the second, note that
  a canonical injection is always a cofibration as it is isomorphic to $1_B
  \sqcup (\initial \rightarrow A)$, and inclusions of the initial object are
  cofibrations by Lemma~\ref{initialTerminalLifting}.  It is a weak equivalence
  because $f$ gives a retraction $B\sqcup A \rightarrow B$.
  
  We now prove that $\tilde f\C = \CC^{core}_{we} \cap \CC^{core}_{fib}$, that
  is, that $\tilde f\C$ is exactly the acyclic fibrations.  

  We first show that $\tilde f\C\subseteq \CC^{core}_{we}\cap \CC^{core}_{fib}$.  By definition, $\tilde f\C\subseteq w\C = \CC^{core}_{we}$.  We also have
  $\tilde f\C = (\CC^{core}_{cof})^\llp \subseteq (\CC^{core}_{cof}\cap
  \CC^{core}_{we})^\llp = \CC^{core}_{fib}$, as desired.  Now let $f:A\rightarrow B \in
  \CC^{core}_{we}\cap \CC^{core}_{fib}$.  As $f\in \CC^{core}_{fib}$, it lifts on the right of $i_1:B
  \stackrel{\sim}{\hookrightarrow} B\sqcup B$. Let $b$ be any morphism $B\rightarrow A$, which exists
  since $f\in \CC^{core}_{we}$, so that we have a commutative diagram
  \[\xymatrix{
    B \ar[rr]^b\cofib[d]_{i_1}^\sim & &  A \fib[d]^{f} \\
    B\sqcup B \ar[rr]_{fb \sqcup 1_B} \ar@{-->}[urr]^{h}_\exists && B }\] 
  This diagram shows that $hi_2$ is a section of $f$, so that $f$ is a retraction
  and therefore $\tilde f\C\supseteq \CC^{core}_{we}\cap \CC^{core}_{fib}$, as desired.

  Now we need to show that $(\CC^{core}_{cof}, \tilde f\C)$ and $(\CC^{core}_{cof} \cap
  \CC^{core}_{we}, \CC^{core}_{fib})$ are WFSs.  We prove this using
  Lemma~\ref{seeWFS}.  $\CC^{core}_{we}$ is closed under retracts because $A\sim B$ is
  an equivalence relation.  $\CC^{core}_{cof}$ and $\CC^{core}_{fib}$ are closed under
  retracts because they are defined by lifing properties, and $\tilde f\C$ is
  closed under retracts because it is equal to $\CC^{core}_{we}\cap \CC^{core}_{fib}$.  Thus
  condition (3) of the lemma holds.  Condition (1) holds by definition of
  $\CC^{core}_{cof}$ and $\CC^{core}_{fib}$.  Thus to show that these are WFSs it suffices to check condition (2).

  First we factor any morphism as a cofibration followed by an acyclic
  fibration. Any morphism $f:A\rightarrow B$ factors as
  \[\xymatrix@C=3em{
    A \cofib[r]^-{i_1}& A\sqcup B \ar@{->>}[r]^-{f\sqcup 1_B}_-{\sim} & B }\] where the
  morphism $i_1$ is a canonical injection into a coproduct (and thus a
  cofibration) and $f\sqcup 1_B$ is a retraction.  This proves condition (2),
  and thus $(\CC^{core}_{cof}, \tilde f\C)$ is a WFS.



  Now we factor any morphism $f:A\rightarrow B$ as an acyclic cofibration
  followed by a fibration.  In particular, we will show that the factorization
  \[\xymatrix@C=3em{
    A \cofib[r]^-{i_1}_-\sim & A\sqcup(A\times B) \ar@{->>}[r]^-{f\sqcup p_2} & B },\]
  where $i_1$ is the canonical injection and $p_2$ is the projection of the
  product on its second factor, works. By our previous analysis we know that $i_1$
  is an acyclic cofibration, so we just need to prove that $f\sqcup p_2:A \sqcup
  (A\times B) \rightarrow B$ is a fibration. Let $e:K\rightarrow L$ be any acyclic
  cofibration and consider any commutative diagram:
  \[\xymatrix{
    K \ar[r]^-k\cofib[d]_{e}^\sim & A\sqcup (A\times B) \ar[d]^{f\sqcup p_2} \\
    L \ar[r]^l & B }\] 
  In order to show that a lift exists, it suffices to show that the lift $h$
  exists in the following diagram:
  \[\xymatrix{
    K \ar[r]^-{i_1}\cofib[d]_{e}^\sim  & K\sqcup L \fib[d]^{e\sqcup 1_L} \ar[rr]^-{k\sqcup (kg\times l)}      && A\sqcup(A\times B) \ar[d]^{f \sqcup p_2} \\
    L \ar[r]^{1_L}\ar@{-->}[ur]^{h} & L \ar[rr]^l && B }\] where $g$ is any morphism from
  $L$ to $K$ (which exists because $e$ is a weak equivalence).  Note that the
  morphism $k\sqcup (kg\times l)$ is not the coproduct of two morphisms, but is
  rather the universal morphism induced by $k$ and $(kg\times l)$.  As $i_2:L
  \rto K\sqcup L$ is a section of $e\sqcup 1_L$, $e\sqcup 1_L \in \tilde f\C$,
  it lifts on the right of $e\in \CC_{cof}^{core}$.  Thus
  $(\CC^{core}_{cof}\cap \CC^{core}_{we}, \CC^{core}_{fib})$ is a WFS, as desired.

  We defer the proof of left and right properness to
  Proposition~\ref{core-proper}.
\end{proof}

Before moving on to prove properness, we need to analyze the cofibrations in
this model structure.  In general, the cofibrations in the core model structure
are very difficult to analyze; however, in the case when $\C$ has splitting and
disjoint coproducts it is possible:
\begin{prop} \label{cofibrationDecomposition} If $\C$ has splitting and disjoint
  coproducts, then any cofibration in the generalized core model structure, $c:A
  \rto B$, is isomorphic to a canonical inclusion $i_1:A \rightarrow A\sqcup X$
  for some object $X$.
\end{prop}

\begin{proof}
  The square 
  \[\xymatrix{A \ar[r]^-{i_1} \cofib[d]_c & A \sqcup B \fib[d]^{c \sqcup 1_B}\\
    B \ar[r]^{1_B} \ar@{-->}[ru]^-h & B }\]
  commutes.
  As $\C$ has splitting coproducts, we can write $h = h_L \sqcup h_R$ with $h_L:
  B_L \rightarrow A$ and $h_R:B_R \rightarrow B$.  Thus $c:A \rightarrow B_L
  \sqcup B_R$, so we can again use splitting to write $c = c_L \sqcup c_R$.  We
  can then rewrite the above diagram as follows:
  \[\xymatrix{A_L\sqcup A_R \ar[rr]^-{i_1} \ar[d]_{c_L\sqcup c_R} && A \sqcup B \ar[d]^{c \sqcup 1_B}\\ %
    B_L\sqcup B_R \ar[rr]^-{\cong} \ar@{-->}[rru]^-{h_L \sqcup h_R} && B }\]
  By considering the restriction to $A_R$ we get that the following diagram
  commutes:
  \[\xymatrix{A_R \ar[d]_{c_R} \ar[rr]^-{i_1} && A_R\sqcup B \\ B_R \ar[r]^{h_R} & B \ar[ru]_{i_2}
  }\]
  Thus $h_Rc_R$ satisfies the conditions of Lemma~\ref{lem:no-trans} and we
  conclude that $A_R = \initial$ and $A_L \cong A$.  Now consider the
  restriction to $A_L$; we get the following diagram:
  \[\xymatrix{A_L \ar[r]^{i_1} \ar[d]_{c_L} & A \ar[d]^c  \\ 
    B_L \ar[r]^{i_1} \ar@{-->}[ru]^{h_L} & B 
  }\]
  As $A_L = A$ we know that $c$ factors through $i_1:B_L \rightarrow B$ as
  $i_1c_L$.  The upper triangle says that $h_Lc_L = 1_A$, and the lower triangle
  and the fact that $c$ factors through $i_1$ says that $i_1c_Lh_L = i_1$; as
  $i_1$ is monic, $c_Lh_L = 1_{B_L}$ and we see that $c_L$ is an isomorphism.
  So we are done.
\end{proof}

We can now prove that the generalized core model structure is left proper and
right proper.

\begin{prop} \label{core-proper}
  In a category with splitting and disjoint coproducts the generalized core
  model structure is left proper and right proper.
\end{prop}
\begin{proof}
  We first need to prove left properness: that the pushout of a weak equivalence
  along a cofibration is a weak equivalence.  By
  Proposition~\ref{cofibrationDecomposition}, we can assume that the cofibration
  is a canonical inclusion $i_1:A\rightarrow A\sqcup C$ and that the weak
  equivalence is $w:A \rightarrow B$; then we have a pushout square
  \begin{equation*}
    \xymatrix{
      A \ar[r]^-{i_1}\ar[d]_{w} & A\sqcup C\ar[d]^{w\sqcup 1_C}\\
      B \ar[r]^-{i_1} & B\sqcup C  
    }
  \end{equation*}
  We want to show that $w\sqcup 1_C$ is a weak equivalence, or in other words
  that there is a morphism $B\sqcup C \rto A\sqcup C$.  As $w$ is a weak
  equivalence there exists a morphism $f:B \rto A$; then $f\sqcup 1_C$ is the
  desired morphism, and we are done.  

  We now consider right properness.  In any model category we can factor a weak
  equivalence as an acyclic cofibration followed by an acyclic fibration.  We
  know that acyclic fibrations are preserved by pullbacks, so in order to show
  right properness it suffices to show that the pullback of an acyclic
  cofibration along a fibration is a weak equivalence.

  By Proposition~\ref{cofibrationDecomposition} we can assume that our
  cofibration is a canonical injection $i_1:A\rightarrow A\sqcup B$. Let $f:C
  \rightarrow A\sqcup B$ be the fibration along which we want to take a
  pullback. By splitting of coproducts, we can write $f=f_L\sqcup f_R$ with
  $f_L:C_L\rightarrow A$ and $f_R:C_R\rightarrow B$. Let $D$ be the pullback of
  our two morphisms, so that we have a diagram
  \[\xymatrix{D \ar[r] \fib[d] & C_L \sqcup C_R \fib[d]^{f_L\sqcup f_R} \\ A \cofib[r]_-{i_1}^-\sim & A\sqcup B 
  }\]
  We want to show that there exists a morphism $g:C_L\sqcup C_R \rightarrow D$.
  Suppose that there exists a morphism $g':C_R \rightarrow C_L$.  Then the
  commutative diagram
  \[\xymatrix@C=4em{C_L \sqcup C_R \ar[d]_{f_L\sqcup g'} \ar[r]^-{1_{C_L}\sqcup f_Lg'} &
    C_L \ar[r]^-{i_1} & C_L \sqcup C_R \ar[d]^{f_L\sqcup f_R} \\ A
    \cofib[rr]_-{i_1}^-\sim && A\sqcup B }\] shows that the morphism $C_L\sqcup
  C_R \rightarrow D$ exists, as desired.  Thus all that we have left to show is
  that $g'$ exists.

  Since $i_1$ is a weak equivalence there exists a morphism $r:B \rto A$.  We
  consider the following diagram, where $h$ exists because $f_L\sqcup f_R$ is a
  fibration and $C_R \rightarrow C_R\sqcup C_R$ is an acyclic cofibration:
  \[\xymatrix@C=5em{
    C_R \ar[r]^{i_2}\cofib[d]_{i_2}^\sim & C_L\sqcup C_R \fib[d]^{f_L\sqcup f_R} \\
    C_R\sqcup C_R \ar[r]_-{(rf_R)\sqcup f_R} \ar@{-->}[ur]^{h} & A\sqcup B }\]
  As we have splitting coproducts, we can write $hi_1: C_R \rto C_L\sqcup C_R$
  as a coproduct of $h_L: X \rto C_L$ and $h_R:Y \rto C_R$.  If we can show that
  there exists a morphism $Y\rto C_L$ we will be done, as $C_R\cong X\sqcup Y$.
  We have a commutative square
  \[\xymatrix{Y \ar[r]^{f_Rh_R} \ar[d]_{rf_Ri_2} & B \ar[d]^{i_2} \\ A \ar[r]^-{i_1} &
    A \sqcup B
  }\]
  As $\C$ has disjoint coproducts the pullback of $i_1$ and $i_2$ is $\initial$;
  thus we have a morphism $Y \rto \initial \rto C_L$ and we are done.
\end{proof}

\subsection*{}

At the beginning of this section we made the choice of setting the acyclic
fibrations to be the retractions.  Instead, we could have taken the dual
definition, and constructed a model structure where the acyclic cofibrations are
the sections:

\begin{thm}
  There is a model structure $\CC^{cocore}$ on $\C$ where $f:A\rightarrow B$ is
  a weak equivalence exactly when $A \sim B$ and the acyclic cofibrations are
  the sections.  If $\C$ has splitting and disjoint coproducts then this model
  structure is right proper; if in addition binary coproducts distribute over
  binary products then it is left proper.
\end{thm}

\begin{proof} 
  The proof that the model structure exists follows by duality from the proof of
  Theorem~\ref{generalizedCoreModel}.  We defer the proof of properness to
  Corollary~\ref{cocore-right} and Proposition~\ref{cocore-left}.
\end{proof}

We call this model structure the {\it generalized cocore model structure}.
Mo\-ral\-ly speaking,  the generalized core and the generalized cocore model
structures should be Quillen equivalent, although we do not know
how to prove this in full generality.  In the case when $\C$ has splitting and
disjoint coproducts, however, this does turn out to be the case:

\begin{prop} \label{core-cocore-equiv} If $\C$ has splitting and disjoint
  coproducts, then the identity functor is a left Quillen equivalence from the
  generalized core model structure to the generalized cocore model structure.
\end{prop}

\begin{proof}
  We know that $\CC_{we}^{core} = \CC_{we}^{cocore}$, so it suffices to show
  that $\CC^{core}_{fib} \supseteq \CC^{cocore}_{fib}$.  Equivalently, it
  suffices to show that $\CC^{core}_{cof} \cap \CC^{core}_{we} \subseteq
  \CC^{cocore}_{cof} \cap \CC^{cocore}_{we}$.

  In the generalized cocore model structure the acyclic cofibrations are
  sections.  In the generalized core model structure the acyclic cofibrations
  are those morphisms $f:A \rightarrow A \sqcup B$ for which a morphism $g:B
  \rightarrow A$ exists.  If such a morphism exists then the induced morphism
  $1_A \sqcup g$ is clearly a retraction for $f$, so all acyclic cofibrations in
  the generalized core model structure have retractions.  Thus all acyclic
  cofibrations in the core model structure are also acyclic cofibrations in the
  cocore model structure, as desired.
\end{proof}

Right properness of the generalized cocore model structure follows directly from
this proposition.

\begin{corollary} \label{cocore-right}
  If $\C$ has splitting and disjoint coproducts then the generalized cocore
  model structure is right proper.
\end{corollary}

\begin{proof}
  By Proposition~\ref{core-cocore-equiv} we know that the identity functor is a
  right Quillen equivalence from the generalized cocore model structure to the
  generalized core model structure; as the two structures have the same weak
  equivalences it suffices to show that the pullback is a weak equivalence in
  the generalized core model structure.  This follows from
  Proposition~\ref{core-proper}.
\end{proof}

To finish the discussion of the cocore model structure, we would like to show
that the generalized cocore model structure is left proper.  However, because of
the way we defined the acyclic fibrations, it turns out to be very difficult to
do so in general.  By introducing a further assumption we get the following
result.

\begin{prop} \label{cocore-left} If $\C$ has splitting and disjoint coproducts
  and, moreover, if binary products distribute over binary coproducts, the
  generalized cocore model structure is left proper.
\end{prop}

\begin{proof}
  Let $c:A\rightarrow B$ be a cofibration. Since in our model structure the
  acyclic cofibrations are the sections, the projections $C\times D\rightarrow
  C$ are fibrations. In particular, $p_2 :(A \sqcup *)\times B\rightarrow B$ is
  a fibration.  However, as products distribute over coproducts we know that
  $(A\sqcup *)\times B \cong (A\times B) \sqcup B$, so there exists a morphism
  $B \rto (A\sqcup *)\times B$.  Thus the morphism $p_2\sqcup 1_B:(A\times B)\sqcup B
  \rightarrow B$ is an acyclic fibration.

  We consider the following commutative diagram and deduce the existence of a
  lifting morphism $h$. 
  \[\xymatrix@C=5em{
    A \cofib[d]_{c}\ar[r]^-{i_1\circ(1_A\times c)} & (A\times B)\sqcup B \fib[d]^{p_2\sqcup 1_B}_\sim \\
    B \ar[r]^{1_B}\ar@{-->}[ur]^{h} & B  
  }\]
  By applying the logic used in Proposition~\ref{cofibrationDecomposition}, we
  see that this diagram is induced from two diagrams
  \[\xymatrix@C=4em{A \ar[r]^-{1_A\times c} \cofib[d]_{c_L} & A\times B
    \fib[d]^{p_2} & & \initial \ar[r] \ar[d] & B \ar[d]^{1_B} \\ 
    B_L \ar[r]^{i_1} \ar@{-->}[ru]^{h_L} & B && B_R  \ar[r]^{i_2} \ar@{-->}[ru]^{h_R}& B
  }\]
  Note that $p_1h_Lc_L = p_1 (1_A\times c) = 1_A$, so $c_L$ is a section.  Thus
  any cofibration is a composition of a section and a canonical inclusion.
  
  Thus it suffices to show that the pushout of a weak equivalence along a
  canonical inclusion or a section is still a weak equivalence.  The pushout of
  a weak equivalence $f:A \rightarrow C$ along a canonical inclusion $i_1:A
  \rightarrow A\sqcup B$ is just $f\sqcup 1_B: A\sqcup B \rightarrow C\sqcup B$,
  which is clearly also a weak equivalence.  The pushout of a section is another
  section, and as all sections are weak equivalences by (2OF3) the pushout of a
  weak equivalence along a section is another weak equivalence, as desired.  So
  we are done.
\end{proof}




We conclude this section with an application of this theorem to the core model
structure defined in \cite{droz12} on the category of finite graphs.  This model
structure agrees with the generalized core model structure defined in
Theorem~\ref{generalizedCoreModel}.

\begin{corollary} \label{graph-disj-split} The categories of finite graphs and
  of infinite graphs have splitting and disjoint coproducts and binary products
  distribute over binary coproducts.  Thus the generalized core and generalized
  cocore model structures on each are both left proper and right proper.  
\end{corollary}

\begin{proof}
  We will show that the categories of graphs have the desired properties; the
  rest follows from the above results.  First we check splitting coproducts.
  Suppose that we have a morphism $f:X \rto A \sqcup B$; this is a map from the
  set of vertices of $X$ to the disjoint union of the vertices of $A$ and $B$.
  Let $X_L$ be the complete subgraph of $X$ on the preimage of the vertices of
  $A$ and let $X_R$ be the complete subgraph on the preimage of the vertices of
  $B$.  $X_L$ and $X_R$ are disjoint subgraphs of $X$ whose union is $X$, so we
  see that $X \cong X_L\sqcup X_R$ and $f = (f|_A: X_L \rto A) \sqcup (f|_B:X_R
  \rto B)$.  Thus we have splitting coproducts.  

  To check that we have disjoint coproducts we just need to check the definition
  on the vertices, where it holds because it holds in the category of sets.

  It remains to show that binary products distribute over binary coproducts.  In
  particular, we want to show that for graphs $A$, $B$ and $C$ we have
  \[A\times(B\sqcup C) \cong (A\times B) \sqcup (A\times C).\] This follows from
  the definitions of graphs and the fact that products distribute over
  coproducts in the category of sets.
\end{proof}

\section{Concepts of cores for infinite graphs} \label{infinite-cores}

We called the model structure constructed in Section~\ref{sec:gencore} the
``generalized core model structure'' because in the case when $\C$ is the
category of finite graphs\footnote{By ``finite graph'' we mean an undirected
  graph with no repeated edges.}, homotopy types correspond exactly to cores.
More precisely, in the model structure two graphs are weakly equivalent exactly
when they have the same core.  (For more on the core, see \cite{godsilroyle},
section 6.2.) Inspired by this, we can consider the generalized core model
structure on the category of all graphs, and ask for a classification of the
homotopy types of this category.  One conjecture is that there should be a notion
of a ``core'' for a (possibly infinite) graph such that cores classify homotopy
types in the generalized core model structure.

Diverse generalizations of the notion of core to infinite graphs have been
explored by Bauslaugh in \cite{bauslaugh}.
\begin{enumerate}
\item An s-core is a graph such that all endomorphisms are surjections (on the
  vertices).
\item An r-core is a graph without proper retractions.
\item An a-core is a graph such that all endomorphisms are automorphisms.
\item An i-core is a graph for which all endomorphisms are injections.
\item An e-core is a graph such that all endomorphisms preserve non-adjacency.
\end{enumerate}
These definitions are known to be equivalent for finite graphs, and are all
proved to be different when considering infinite graphs in \cite{bauslaugh}.

Once a definition of core is chosen, we can define {\it a core of a graph} $G$
as one of its subgraphs $H$, which is a core and for which a morphism
$G\rightarrow H$ exists. It would also make sense to define the core as a
retraction of $G$; however, as this definition is more restrictive than the
previous one, the results of this section will also hold under this definition.

It is natural to ask if applying our generalized core construction to the
category of all graphs gives a notion of core that corresponds to one of those
defined above.  More precisely, is it the case that two graphs are weakly
equivalent if they have the same ``core'', for some notion of ``core'' defined
above?  The answer turns out to be ``no.''  We prove this by exhibiting two
graphs, one of which does not contain a core in the sense of (1)-(3), and one of
which does not contain a core in the sense of (3)-(5).  As every graph has a
``homotopy type'' in the generalized core model structure, this means that none
of these definitions of a core classify homotopy types in the case of the
generalized core model structure.  

Note that while the definitions above were originally given for general
``structures'' (understood as combinatorial structures), and exemplified by
oriented graphs, it can be shown (\cite{pultr}) that all of the relevant
examples and results can be transferred to the category of undirected graphs
using a well-chosen fully faithful ``edge-replacement'' functor.  Thus it
suffices to show that there exist directed infinite graphs with no core, and it
will also hold for undirected graphs.

We construct our examples by adapting methods from \cite{bauslaugh}.

\begin{figure}
\centering
\includegraphics[trim = 0mm 140mm 0mm 0mm, clip, width=0.8\textwidth]{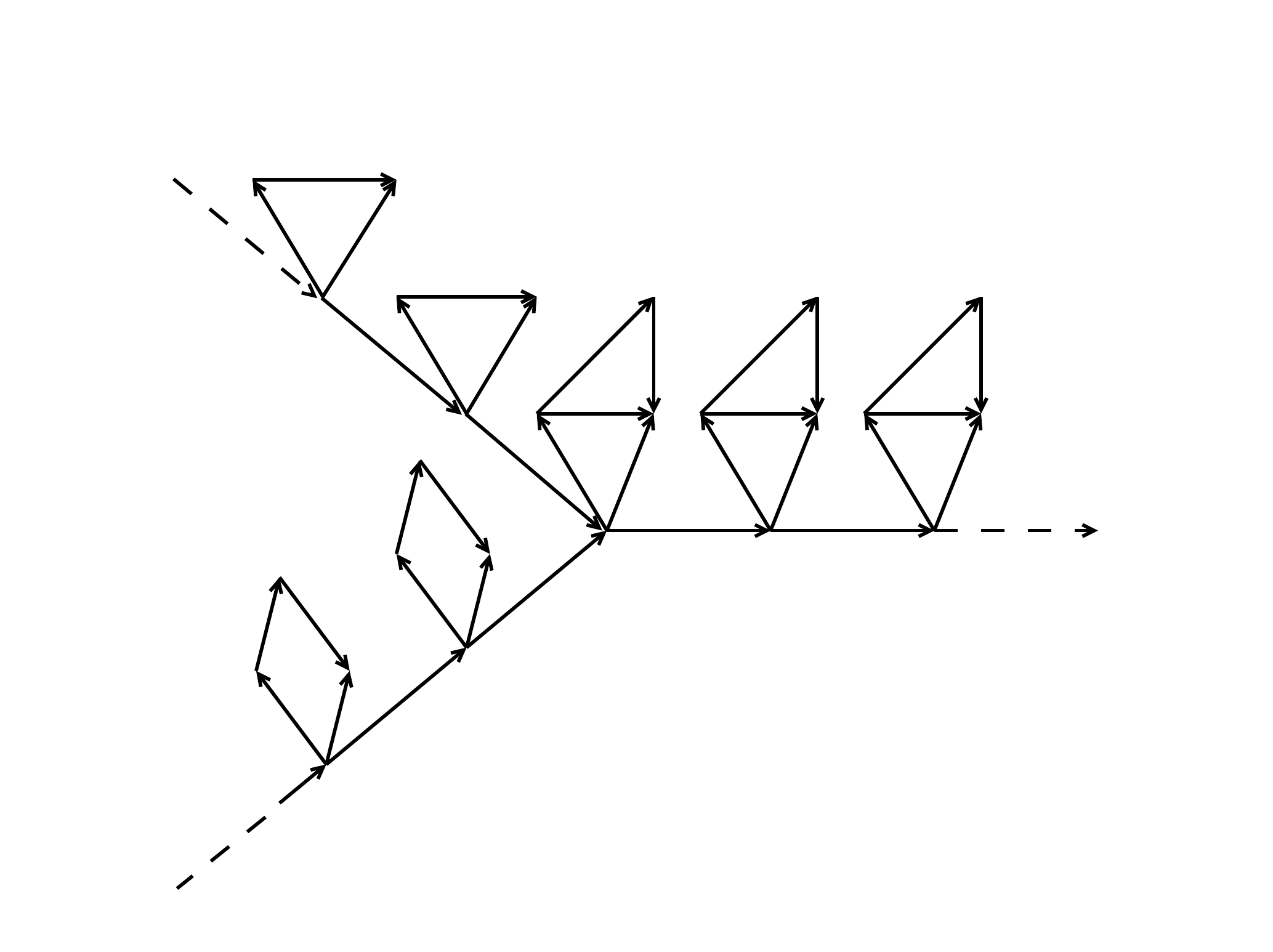}
\caption{Prolongating this graph in three directions without end, we obtain the {\it zipper} graph.}
\label{fig:zipperGraph}
\end{figure}

\begin{thm}
  Let $G$ be the graph with vertices $\{1,2,\ldots\}$ and with an edge from $n$
  to $n+1$ for all $n$.  Then $G$ has no s-core, r-core or a-core. The zipper
  graph in Figure~\ref{fig:zipperGraph} has no a-core, i-core or e-core.
\end{thm} 

\begin{proof}
  Any endomorphism $\varphi$ of $G$ is uniquely determined by $\varphi(1)$, and
  must have an image isomorphic to itself.  Thus $G$ has a core if and only if
  it is a core.  However, it is clearly not an s-core, an r-core or an a-core,
  and thus $G$ has none of these cores.

  The zipper graph is composed of three infinite rays with a common point, two
  of the rays going to the common point, one ray coming out of the common point
  and additional decorations. We observe that the endomorphisms of the zipper
  graph map the outgoing ray to itself by a ``shift toward the right''. The
  decorations insure the absence of an automorphism mapping one of the incoming
  rays to the other. Since the non-trivial endomorphisms of the zipper graph are
  non-injective but surjective, the zipper graph has no i-core or
  a-core. Moreover, looking at non-adjacent vertices of the decorations of the
  lower incoming ray, we see that they can sometimes be mapped to adjacent
  vertices. This shows that the zipper graph has no e-core and concludes the
  proof of our theorem.
\end{proof} 

We conclude that the generalized core model structure has a notion of homotopy
type which does not correspond to any of Bauslaugh's definition of cores.


\bibliographystyle{alpha} 
\bibliography{JMD-IZ}
\end{document}